\documentclass[final]{IEEEtran}

\usepackage{amsthm,amssymb,latexsym,color,amsmath,pifont,epsfig,graphicx}
\usepackage{setspace}
\usepackage{subfig}  %Use as needed
\usepackage{url}

\usepackage{algorithm}
\usepackage{algorithmic}

\usepackage{tikz}
\usepackage{pgfplots}
\usepackage{dblfloatfix}
\usepackage{pifont,color}
\definecolor{darkgreen}{RGB}{0,100,0}   %define a custom color

\newtheorem{definition}{Definition}%[section]
\newtheorem{lemma}{Lemma}%[section]
\newtheorem{proposition}{Proposition}%[section]
\newtheorem{remark}{Remark}%[section]
\newcommand{\R}{\mathbb{R}}
\newcommand{\DefinedAs}[0]{\mathrel{\mathop:}=}
\newcommand{\mini}{\mathop{\mbox{minimize}}}
\newcommand{\ds}{\displaystyle}

\newcommand{\st}{\mbox{subject to}}

\newcommand{\amin}{\mathop{\mbox{argmin}}}

\newcommand{\card}{{\bf card}}
\newcommand{\rank}{{\bf rank}}

\newcommand{\bff}{{\bf f}}

\newcommand{\bfh}{{\bf h}}
\newcommand{\bfu}{{\bf u}}

\newcommand{\tc}{\textcolor}

\begin{document}

\title{Learning Low-Complexity Autoregressive Models via Proximal Alternating Minimization}

\author{Fu Lin and Jie Chen
\thanks{F.\ Lin is with the Systems Department, United Technologies Research Center, 411 Silver Ln, East Hartford, CT 06118, USA. E-mail: linf@utrc.utc.com}
\thanks{J.\ Chen is with IBM Thomas J. Watson Research Center, 1101 Kitchawan Road, Yorktown Heights, NY 10598, USA.  E-mail: chenjie@us.ibm.com}
}

\maketitle

    \begin{abstract}
We consider the estimation of the state transition matrix in vector autoregressive models, when time sequence data is limited but nonsequence steady-state data is abundant. To leverage both sources of data, we formulate the least squares minimization problem regularized by a Lyapunov penalty. We impose cardinality or rank constraints to reduce the complexity of the autoregressive model. We solve the resulting nonconvex, nonsmooth problem by using the proximal alternating linearization method (PALM). We show that PALM is globally convergent to a critical point and that the estimation error monotonically decreases. Furthermore, we obtain explicit formulas for the proximal operators to facilitate the implementation of PALM. We demonstrate the effectiveness of the developed method on synthetic and real-world data. Our experiments show that PALM outperforms the gradient projection method in both computational efficiency and solution quality.
    \end{abstract}

%\begin{keyword}

\section{Introduction}
\label{sec.intro}

Vector autoregressive (VAR) models are widely used in the analysis of linear interdependence in time series data. A key step in building the VAR model is the identification of the state transition matrix. When time sequence data is adequate, the standard approach is to solve a least-squares problem. In modern applications, however, the dimension of the model is significantly larger than the number of time sequence measurements, which makes the model unidentifiable through the standard least-squares approach. Such scenarios include, for example, tracking the progression of brain neurological diseases, because the number of comprehensive brain scans is limited due to cost or medical concerns~\cite{huasch11}. In gene expression networks, the number of genes is typically much larger than the number of measurements, because of the intrusive nature of the measuring techniques~\cite{yosimohig05,zavjulboy11,wanhursch13}. 

In such situations, regularization is a typical rescue. For example, ridge regularization is a common approach for ensuring a unique solution. Other regularization approaches introduce additional structures to the solution. In particular, sparsity and low-rank structures are extensively studied. These regularization approaches are popular, in part because the resulting problem may be efficiently solved by using convex optimization techniques~\cite{fujsatcar07,wanlitsa07,gupbar08,huasch09,zavjulboy11,huasch11,hanliu13a,bahliuxin13,geizhasch15}.  In~\cite{fujsatcar07}, a sparse VAR model is found via Lasso for gene regulatory networks. In~\cite{bahliuxin13},  the state transition matrix is decomposed into a sparse matrix and a low-rank matrix by using convex penalty functions. Other approaches based on convex optimization can be found in~\cite{wanlitsa07,gupbar08,huasch09,zavjulboy11,huasch11,hanliu13a,geizhasch15}. 

In a different vein, steady-state data provide opportunity for improving model accuracy. When the VAR model is stable and steady-state data are abundant, several authors show that the steady-state data can help reduce the estimation error~\cite{huasch09,huasonsch10,huasch11,zavjulboy11,wanhursch13,largorhom15}. In~\cite{huasch11}, steady-state data are leveraged to form the Lyapunov regularization. In~\cite{zavjulboy11}, the perturbed steady-state data is used to infer sparse, stable gene expression networks. In~\cite{wanhursch13}, both steady-state and temporal data are integrated in the estimation of the gene regulatory networks. Other work that employs steady-state data for system identification includes~\cite{huasch09,huasonsch10,largorhom15}.

In this paper, we leverage both time sequence and steady-state nonsequence data for the model estimation. We propose a least-squares estimator regularized by the Lyapunov penalty subject to the cardinality or rank constraints on the state transition matrix. The identification problem is nonconvex due to the Lyapunov penalty and nonsmooth due to the low-complexity constraints. We solve the problem by using the proximal alternating linearization method~(PALM). An advantage of PALM is that it converges to a critical point starting from any initial condition. We prove this global convergence property of PALM and show that the estimation error is monotonically decreasing with the PALM iterations. We obtain closed-form expressions for the proximal operators to facilitate implementation. We show that PALM can handle the stability constraints and also the convex low-complexity (e.g., the $\ell_1$ or the nuclear-norm) constraints. We demonstrate that our approach outperforms the gradient projection method in both computational time and solution quality.

Our presentation is organized as follows. In Section~\ref{sec.var}, we formulate the estimation problem for the low-complexity VAR model. In Section~\ref{sec.PALM}, we present the PALM algorithm and derive explicit formulas for the proximal operators. In Section~\ref{sec.conv}, we show the global convergence of PALM by establishing the Lipschitz conditions and the KL property of the estimation problem. In Section~\ref{sec.exp}, we demonstrate the effectiveness of PALM via numerical experiments. In Section~\ref{sec.concl}, we summarize our contributions and discuss future directions.

\section{Model Identification via Lyapunov Penalty}
\label{sec.var}

In this section, we formulate the model identification problem using both time-sequence data and steady-state data. The performance of the model is measured by the least-squares error for the time-sequence data and the Lyapunov penalty for the steady-state data. We employ low-complexity penalty functions to promote sparsity and low-rank properties of the state transition matrix.

Consider a $p$-dimensional vector autoregressive model:
\begin{equation}
  \label{eq.ar}
  \phi (t+1) \,=\, A \phi (t) \,+\, \epsilon(t),
\end{equation}
where $\phi(t) \in \R^p$ is the state vector, $A \in \R^{p \times p}$ is the state transition matrix, and $\epsilon(t) \in \R^p$ is a zero-mean white stochastic process. We assume that the autoregressive model~\eqref{eq.ar} is asymptotically stable; that is, all eigenvalues of $A$ have modulus less than one. The state vector $\phi(t)$ has a steady-state distribution, whose covariance matrix $P$ is determined by the discrete-time Lyapunov equation
\[  
  A P A^T \,+\, Q \;=\; P,
\]
where $Q \in \R^{p \times p}$ is the covariance matrix of $\epsilon(t)$. Linear systems theory says that $P$ is positive definite if and only if $A$ is asymptotically stable~\cite{dulpag13}.

Our objective is to identify the state transition matrix $A$. Given a set of $n$ time sequence measurements of $\phi(t)$, the standard least-squares estimation is given by
\begin{equation}
  \label{eq.ls}
\mini_{X \in \R^{p \times p}} \; \frac{1}{2} \| X \Phi \,-\, \Psi \|_F^2 ,
\end{equation}
where $\Phi \DefinedAs [\,\phi(1), \cdots, \phi(n-1)\,] \in \R^{p \times (n-1)}$, $\Psi \DefinedAs [\,\phi(2), \cdots, \phi(n)\,] \in \R^{p \times (n-1)}$, and $\|\cdot\|_F$ denotes the Frobenius norm. We use $X$ to denote the unknown state transition matrix for the convenience of developing optimization details. When the number of time sequence data is less than the dimension of the states (i.e., $p > n - 1$), infinitely many solutions exist for~\eqref{eq.ls} and the state transition matrix is unidentifiable. 

We are interested in the scenario when the time sequence data is scarce but the steady-state nonsequence data is readily available~\cite{huasch09,huasonsch10,huasch11,gupbar08,lozabeliuros09}. In this case, Huang and Schneider~\cite{huasch11} propose the Lyapunov penalty as a regularization term
\begin{equation}
\label{eq.lyap}
 \| X P X^T \,+\, Q \,-\, P \|_F^2 .
\end{equation}
They show that the Lyapunov penalty helps improve the accuracy of the estimation~\cite{huasch11}. Since the covariance matrix $P$ is unknown, we replace it by the sample covariance 
\[
  S \; \DefinedAs \;  \frac{1}{N} \sum_{i=1}^N (z^i - \bar{z}) (z^i - \bar{z})^T
  ~\mbox{with}~
  \bar{z} \; \DefinedAs \; \frac{1}{N} \sum_{i=1}^N z^i,
\]
where $\{z^i\}_{i=1}^N$ is the steady-state nonsequence data. The identification problem with the Lyapunov regularization can be expressed as
\begin{equation}
  \label{eq.hs11}
\mini_{X  \in \R^{p \times p}} \; \frac{1}{2} \| X \Phi \,-\, \Psi  \|_F^2
          \,+\, \frac{\rho}{2}  \|   X S X^T \,+\, Q \,-\, S \|_F^2 ,
\end{equation}
where $\rho$ is a positive coefficient that balances the estimation error between the sequence and the nonsequence data. 

\tc{black}{Huang and Schneider study~\eqref{eq.hs11} and show that the Lyapunov penalty improves the solution quality. However, there is no guarantee that the solution of~\eqref{eq.hs11} is stable (i.e., spectral radius of $X$ is less than 1). We next incorporate stability constraint into~\eqref{eq.hs11}.}

\subsection{Stability Constraint}
\label{sec.stab}

Since stability is a necessary condition for the use of Lyapunov penalty~\eqref{eq.lyap}, we impose a stability constraint in the identification problem~\eqref{eq.hs11}. Let $\tau(X)$ denote the spectral radius of $X$, that is, $\tau(X) \DefinedAs \max\{|\lambda_i|\}_{i=1}^p$. \tc{black}{A stable autoregressive model can be obtained by solving the following problem:
\begin{equation}
  \label{eq.lcstab}
  \begin{array}{ll}
      \ds
      \mini_{X \in \R^{p \times p}} & \; \ds \frac{1}{2} \| X \Phi \,-\, \Psi  \|_F^2
              \,+\, \ds \frac{\rho}{2}  \|   X S X^T  \,+\, Q \,-\, S \|_F^2 \\
       \st & \;
       \tau(X) \,<\, 1.
  \end{array}
\end{equation}
Dealing with $\tau$ directly is difficult because spectral radius is neither convex nor locally Lipschitz~\cite{ovewom88,kimgupos09}. Alternatively, one can employ a convex function as an upper bound~\cite{xiaboy04}.
Since $\tau(X) \leq \|X\|_2 \leq \|X\|_F$ (see~\cite[Chapter 5]{horjoh12}), we can incorporate the stability constraint in the cost function 
\begin{equation}
  \label{eq.lcstabw}
  %\begin{array}{r}
      \ds
      \mini_{X \in \R^{p \times p}}  \,\, \ds \frac{1}{2} \| X \Phi \,-\, \Psi  \|_F^2
              \,+\, \ds \frac{\rho}{2}  \|   X S X^T  \,+\, Q \,-\, S \|_F^2 
      \,+\, \dfrac{\mu}{2} \| X X^T \|^2_F
  %\end{array}
\end{equation}
where $\mu$ is a positive constant. While one can employ the spectral norm $\|X\|_2$ as a less conservative proxy, we choose the Frobenius norm mainly because  both the Lyapunov penalty~\eqref{eq.lyap} and the stability penalty $\| X X^T \|_F^2$ are then quadratic functions of $X$ in Frobenius norm squared.} Hence, the stability term $\| X X^T \|_F^2$ is inconsequential in the design of solution methods. For this reason and for the ease of presentation, in what follows we omit the stability penalty, but comment on the modification of the algorithm when appropriate to address stability. \tc{black}{Detailed analysis of spectral radius and its relaxation in minimization problem can be found in~\cite{ovewom88,burove01,xiaboy04,kimgupos09}.}

\subsection{Low-Complexity Models}

In several applications, it is desired to impose sparsity or low-rank structures on the state transition matrix~\cite{fujsatcar07,wanlitsa07,gupbar08,huasch09,zavjulboy11,huasch11,hanliu13a,bahliuxin13,geizhasch15}. In gene expression networks, for example, the nonzero elements of the state transition matrix determine the interaction graph of the expression network~\cite{fujsatcar07,zavjulboy11}. A sparse state transition matrix is useful because one can construct a sparse network to explain experiment data. 

One common approach to promoting sparsity is to impose the $\ell_1$ constraint:
\begin{equation}
  \label{eq.l1}
\|X\|_{\ell_1} \DefinedAs \sum_{i,j=1}^p |X_{ij}| \, \leq \, l,
\end{equation}
where $l$ is a prescribed positive number.  Since the $\ell_1$ norm promotes sparsity~\emph{implicitly}, the actual number of nonzero elements in the solution is indirectly controlled by the threshold $l$. However, given a desired level of sparsity, the correct choice of $l$ is typically unknown a priori. An \emph{explicit} way to guarantee sparsity is to control the number of nonzero elements by the cardinality constraint:
\begin{equation}
  \label{eq.cardc}
  \card(X) \, \DefinedAs ~\mbox{number of nonzero entries of}~ X \, \leq \, s,
\end{equation}
where  $s$ is a given positive integer. Note that the cardinality constraint is harder to deal with than the $\ell_1$ constraint, because cardinality is a nonconvex function.

Another approach to obtaining low-complexity models is to impose the low-rank constraint. A low-rank state transition matrix is useful because it implies that the data can be explained by a  model with lower dimensions. An implicit way to promote low-rank solutions is to use the nuclear norm constraint~\cite{fazhinboy01,liuvan09,recfazpar10,fazponsuntse13}
\begin{equation}
\label{eq.nuc}
\| X \|_*  \, \DefinedAs \, \sum_{i=1}^p \sigma_i (X)  \, \leq \, \nu,
\end{equation}
where $\nu$ is a prescribed positive number and the $\sigma_i$s are the singular values. Similar to the sparsity case, the threshold $\nu$ is not known a priori. We impose a low-rank constraint by controlling the rank of the state transition matrix:
\begin{equation}
  \label{eq.rankc}
  \rank(X) \, \DefinedAs ~\mbox{number of nonzero singular values of}~ X \, \leq \, r,
\end{equation}
where $r$ is a given positive integer.

Hence we consider the following estimation problem:
\begin{equation}
  \label{eq.lc}
  \begin{split}
      \widehat{A}\,=\,&
      \ds
      \amin_{X \in \R^{p \times p}} \,\, \ds \frac{1}{2} \| X \Phi \,-\, \Psi  \|_F^2
              \,+\, \ds \frac{\rho}{2}  \|   X S X^T  \,+\, Q \,-\, S \|_F^2 \\
      & \st \,\,
      \mbox{constraint} ~ \eqref{eq.l1} ~\mbox{or}  ~\eqref{eq.cardc}~ \mbox{or}  ~\eqref{eq.nuc}~ \mbox{or} ~\eqref{eq.rankc}.
  \end{split}
\end{equation}
For the convex constraints~\eqref{eq.l1} and~\eqref{eq.nuc}, one may employ gradient projection methods; namely, taking a descent direction of the objective function and projecting it onto the convex constraint sets. A gradient projection method is proposed in~\cite{huasch11} to solve~\eqref{eq.lc} with the $\ell_1$ constraint~\eqref{eq.l1}. For the nonconvex constraints~\eqref{eq.cardc} and~\eqref{eq.rankc}, on the other hand, we develop the PALM algorithm in the subsequent section.

\section{Proximal Alternating Linearized Method}
\label{sec.PALM}

In this section, we develop the PALM algorithm for the identification problem of low-complexity models. This approach decomposes the problem into a sequence of smaller problems that can be solved efficiently. Furthermore, we show that PALM is global convergence to a critical point for both convex and nonconvex constraints in~\eqref{eq.lc}. 

We begin with a reformulation of the low-complexity autoregressive models~\eqref{eq.lc}
\begin{equation}
  \nonumber
  \begin{array}{ll}      
    \ds  \mini_{X, Y \in \R^{p \times p}} 
       &  \ds \frac{1}{2} \| X \Phi \,-\, \Psi  \|_F^2
              \,+\, \ds \frac{\rho}{2}  \|   Y S X^T  \,+\, Q \,-\, S \|_F^2 \\
      \st  &  Y \, - \, X \,=\, 0, \\
      &  \eqref{eq.l1}~ \mbox{or}  ~\eqref{eq.cardc} ~\mbox{or} ~\eqref{eq.nuc} ~\mbox{or} ~\eqref{eq.rankc}
  \end{array}
\end{equation}
where we replace one of the two $X$s in the Lyapunov penalty by a new variable $Y$. Let $f$ denote the least-squares term 
\begin{equation}
\label{eq.f}
f(X) \,=\, \frac{1}{2} \| X \Phi \,-\, \Psi \|_F^2,
\end{equation}
and let $g$ denote the indicator function of the individual constraints in~\eqref{eq.l1}-\eqref{eq.rankc}, for example,
\begin{equation}
  \label{eq.card}
  g(Y) \,=\,
  \left\{
    \begin{array}{ll}
      0, & \card(Y) \,\leq\, s \\
     \infty, & \mbox{otherwise}
    \end{array}
  \right.
\end{equation}
for the cardinality constraint~\eqref{eq.cardc} and 
\begin{equation}
  \label{eq.rank}
  g(Y) \,=\,
  \left\{
    \begin{array}{ll}
      0, & \rank(Y) \,\leq \,r \\
     \infty, & \mbox{otherwise}
    \end{array}
  \right.
\end{equation}
for the rank constraint~\eqref{eq.rankc}. Then we have 
\begin{equation}
   \label{eq.lcstd}
  \mini_{X,Y \in \R^{p \times p}} \; \omega(X,Y) \,\DefinedAs \, f(X) \,+\, g(Y) \,+\, h(X,Y) ,
\end{equation}
where $h$ denotes the coupling term 
\begin{equation}
  \label{eq.H}
  h(X,Y) \,=\, \frac{\rho_1}{2} \| Y S X^T \,+\, Q \,-\, S \|_F^2
  \,+\, \frac{\rho_2}{2} \|X \,-\, Y\|_F^2.
\end{equation}
Here, the penalty parameter $\rho_1>0$ resumes the role of $\rho$ in~\eqref{eq.lc} and $\rho_2>0$ is sufficiently large to penalize the discrepancy between $X$ and $Y$. \tc{black}{It is worth mentioning that the convergence of PALM does not depend on the choice of $\rho_1$ and $\rho_2$. This is in contrast to ADMM that may require sufficiently large quadratic term to ensure convergence when it is applied to nonconvex problems~\cite{honluoraz16,hajchawan16}.}

\subsection{Generic PALM Method}

PALM computes the proximal operators of the {\em uncoupled\/} functions $f$ and $g$, around the linearization of the {\em coupling\/} function $h$ at the previous iterate, hence the name~\cite{attbolred10,attbolsva13,parboy13,bolsabteb14}. It is instructive to put PALM in the context of other alternating methods. Suppose for the moment that $\omega(X,Y)$ is a strictly convex function. One approach to minimizing $\omega$ is the Gauss-Seidel iteration (also known as the coordinate descent):
\begin{subequations}
  \nonumber
  \begin{align}
    X^{k+1} & \;\in\; \amin_X ~ \omega(X,Y^k) \\
    Y^{k+1} & \;\in\; \amin_Y ~ \omega(X^{k+1},Y) .
  \end{align}
\end{subequations}
Convergence of the iteration requires a unique solution in each minimization step; otherwise, Gauss-Seidel may cycle indefinitely~\cite{pow73}. When $\omega$ is convex but {\em not strictly\/} convex, uniqueness can be achieved by including a quadratic proximal term
\begin{subequations}
  \begin{align}
  \label{eq.proxx}
    X^{k+1} & \;\in\;  \amin_X \left\{ \omega(X, Y^k) \,+\, \dfrac{c_k}{2} \| X \,-\, X^k \|^2_F \right\} \\
  \label{eq.proxy}
    Y^{k+1} & \;\in\; \amin_Y \left\{ \omega(X^{k+1}, Y) \,+\, \dfrac{d_k}{2} \| Y \,-\, Y^k \|^2_F \right\}     ,
  \end{align}
\end{subequations}
where $c_k$ and $d_k$ are positive coefficients. This class of proximal methods is well studied; see~\cite{parboy13} for a recent survey. 

When $\omega$ is nonconvex, as in our case~\eqref{eq.lcstd}, we need to modify the proximal terms to ensure convergence. Instead of taking the proximal term around $X^k$ as in~\eqref{eq.proxx}, we take the term around $X^k$ modified with a scaled partial gradient of $h$:
\begin{equation}
  \label{eq.proxf}
  X^{k+1} \; \in \;
  \amin_X \left\{ f(X) \,+\, \frac{c_k}{2} \| X \,-\, U^k \|_F^2 \right\}  ,
\end{equation}
where
$
  U^k \,=\, X^k \,-\, \frac{1}{c_k} \nabla_X h(X^k,Y^k).
$
The parameter $c_k$ is chosen to be greater than the Lipschitz constant of $\nabla_X h$; in particular,
$
	c_k \,=\, \gamma_1 L_1(Y^k) 
$
for some $\gamma_1 > 1$ where $L_1$ is the Lipschitz constant of $\nabla_X h$.

Similarly, we take the proximal term around $Y^k$ modified with a scaled partial gradient of $h$:
\begin{equation}
  \label{eq.proxg}
  Y^{k+1} 
  \; \in \;
  \amin_Y \left\{ g(Y) \,+\, \frac{d_k}{2} \| Y \,-\, V^k \|_F^2 \right\}  ,
\end{equation}
where
$
  V^k \,=\, Y^k \,-\, \frac{1}{d_k} \nabla_Y h (X^{k+1},Y^k).
$
The parameter $d_k$ is determined by $d_k = \gamma_2 L_2(X^{k+1})$ for some $\gamma_2 > 1$ where $L_2$ is the Lipschitz constant of $\nabla_Y h$. PALM alternates between updating $(X,Y)$ by using the iterations~\eqref{eq.proxf}-\eqref{eq.proxg}.

\subsection{Formulas for Lipschitz Constants and Solutions to~\eqref{eq.proxf}-\eqref{eq.proxg}}

To implement~\eqref{eq.proxf}-\eqref{eq.proxg}, one needs the Lipschitz constants $L_1$ and $L_2$ in order to determine the coefficients $c_k$ and $d_k$, respectively. Taking the partial gradients of $h$ yields
\[
  \begin{array}{l}
    \nabla_X h \,=\, \rho_1( X S^T Y^TY S  \,+\, (Q \,-\, S)^T Y S) \,+\, \rho_2(X \,-\, Y) \\
    \nabla_Y h \,=\, \rho_1(Y S X^TX S^T  \,+\, (Q \,-\, S) X S^T) \,+\, \rho_2(Y \,-\, X) .
  \end{array}
\]
%\[
%  \begin{array}{l}
%    \nabla_X H = \rho_1( X S^T Y^TY S  + (Q \!-\! S)^T Y S) + \rho_2(X \!-\! Y) \\[0.1cm]
%    \nabla_Y H = \rho_1(Y S X^TX S^T  + (Q \!-\! S) X S^T) + \rho_2(Y \!-\! X) .
%  \end{array}
%\]
Since $\nabla_X h$ is linear in $X$ and $\nabla_Y h$ is linear in $Y$, we obtain explicit formulas for the Lipschitz constants 
\begin{equation}
  \label{eq.lipcon}
  \begin{array}{l}
    L_1(Y) \,=\, \ds \| \rho_1 S^T Y^T Y S  \,+\,  \rho_2 I\|_2  \\[0.1cm]
    L_2(X) \,=\, \ds \| \rho_1 S X^T X S^T  \,+\, \rho_2 I \|_2
  \end{array}
\end{equation}
where $\| \cdot \|_2$ denotes the largest singular value of a matrix.

We next show that the proximal operators~\eqref{eq.proxf}-\eqref{eq.proxg} can be computed efficiently. The proximal operator~\eqref{eq.proxf} can be expressed as
\[
    X^{k+1} \; \in \;   \amin_X \left\{ \frac{1}{2} \|X \Phi \,-\, \Psi \|_F^2 \,+\, \frac{c_k}{2} \| X \,-\, U^k \|_F^2 \right\}.
\]
Solving this least-squares problem yields
\[
X^{k+1} \;=\; ( \Psi \Phi^T + c_k U^k) ( \Phi \Phi^T \,+\, c_k I)^{-1},
\]
where $I$ denotes the identity matrix. When the number of states is no less than the number of time sequence data (i.e., $p \geq n$),  one can reduce the computational cost by inverting $\Phi^T \Phi \,+\, c_k I$ instead of $\Phi \Phi^T \,+\, c_k I$, since the Woodbury formula gives 
\[
X^{k+1} \;=\;  (c_k^{-1} \Psi \Phi^T + U^k)(I - \Phi ( c_k I + \Phi^T \Phi)^{-1} \Phi^T).
\]

The proximal operator~\eqref{eq.proxg} can be expressed as
\[
\begin{array}{ll}
\ds \mini_Y  &  \dfrac{d_k}{2} \| Y \,-\, V^k \|_F^2
\\[0.1cm]
\st & \eqref{eq.l1} ~\mbox{or} ~\eqref{eq.cardc} ~\mbox{or} ~\eqref{eq.nuc} ~\mbox{or} ~\eqref{eq.rankc}.
\end{array}
\]
For the cardinality constraint~\eqref{eq.cardc}, the solution is obtained by keeping the $s$ largest elements of $V^k$ in magnitude and zero out the rest of the elements in $V^k$. This is because the squared Frobenius norm is the sum of the squared elements of $Y-V^k$. For the rank constraint~\eqref{eq.rankc}, by the Eckart--Young theorem, the solution is the best rank-$r$ approximation of $V^k$ obtained by the truncated SVD; that is, keeping the $r$-largest singular value and setting the remaining singular values of $V^k$ to zero. 

For the $\ell_1$ constraint~\eqref{eq.l1}, the projection onto the $\ell_1$-ball can be computed by an algorithm developed in~\cite{ducshasin08}. For the nuclear-norm constraint~\eqref{eq.nuc}, the optimal solution $Y$ can be computed by performing the singular value decomposition of $V^k$ and then projecting the singular values of $V^k$ onto the $\ell_1$-ball. 

%For completeness, this algorithm is included in~\ref{sec.l1proj}. 

We summarize the computational steps in Algorithm~\ref{alg.palm}, focusing on only the constraints~\eqref{eq.cardc} and~\eqref{eq.rankc}. 
\begin{algorithm*}[htb]
   \caption{Proximal Alternating Linearization Method for~\eqref{eq.lcstd}}
   \label{alg.palm}
\begin{algorithmic}
   \STATE Initialization: Start with any $(X^0,Y^0)$.
   \FOR{$k=0,1,2,\ldots$ until convergence}
   \STATE $\triangleright$ The following section computes $X^{k+1}$
   \STATE Compute the Lipschitz constant 
   $L_1(Y^k) \,=\, \| \rho_1 S^T Y^{kT} Y^k S  +  \rho_2 I\|_2$.
   \STATE Compute $c_k \,=\, \gamma_1 L_1(Y^k)$ for some $\gamma_1 > 1$.
   \STATE Compute the partial gradient \\
   $\nabla_X h(X^k,Y^k) \,=\, \rho_1( X^k S^T Y^{kT} Y^k S  \,+ \,(Q - S)^T Y^k S) \,+\, \rho_2(X^k - Y^k)$.
   \STATE Update the proximal point 
   $U^k \,=\, X^k - \frac{1}{c_k} \nabla_X h(X^k,Y^k)$.
   \IF{$p<n$}
   \STATE $X^{k+1} \,=\, ( \Psi \Phi^T + c_k U^k) ( \Phi \Phi^T + c_k I)^{-1}$
   \ELSE
   \STATE $X^{k+1} \,=\,  (c_k^{-1} \Psi \Phi^T + U^k)(I - \Phi ( c_k I + \Phi^T \Phi)^{-1} \Phi^T)$.
   \ENDIF
   \STATE $\triangleright$ The following section computes $Y^{k+1}$
   \STATE Compute the Lipschitz constant \\
   $L_2(X^{k+1}) \,=\, \| \rho_1 S X^{(k+1) T} X^{k+1} S^T  + \rho_2 I \|_2$.
   \STATE Compute $d_k \,=\, \gamma_2 L_2(X^{k+1})$ for some $\gamma_2 > 1$.
   \STATE Compute the partial gradient \\
   $\nabla_Y h(X^{k+1},Y^k) \,=\, \rho_1(Y^k S (X^{k+1})^T X^{k+1} S^T  \,+ \,(Q - S) X^{k+1} S^T) \,+\, \rho_2(Y^k - X^{k+1})$.
   \STATE Update the proximal point 
   $V^k \,=\, Y^k - \frac{1}{d_k} \nabla_Y h(X^{k+1},Y^k)$.
   \IF{$g$ is the cardinality constraint~\eqref{eq.cardc}}
   \STATE $Y^{k+1} \,=\, \mathcal{I}_s\circ V^k$, where $(\mathcal{I}_s)_{ij}=1$ if $(|V^k|)_{ij} \geq $ $s$-th largest element of $|V^k|$, and $(\mathcal{I}_s)_{ij} = 0$ otherwise.
   \ELSIF{$g$ is the rank constraint~\eqref{eq.rankc}}
   \STATE $Y^{k+1}$ is the rank-$r$ truncated SVD of $V^k$.
   \ENDIF
   \ENDFOR
\end{algorithmic}
\end{algorithm*}

We conclude this section with a remark on stability.
\begin{remark}[Stability] 
As discussed in Section~\ref{sec.stab}, we can incorporate the stability constraint by penalizing $\| X X^T \|_F^2$ in the cost function. In this case, the coupling term becomes
\[
\begin{array}{rcl}
  h(X,Y) &=& \dfrac{\rho_1}{2} \| Y S X^T + Q - S \|_F^2
  \,+\, \dfrac{\rho_2}{2} \|X - Y\|_F^2 
  \\
  && + \dfrac{\mu}{2} \| Y X^T \|_F^2.
\end{array}
\]
Its partial gradients are given by
\[
  \begin{array}{rcl}
    \nabla_X h &=& \rho_1( X S^T Y^T Y S  + (Q -  S)^T Y S) + \rho_2(X - Y)  
    \\ 
                         && + \mu X Y^T Y     \\[0.1cm]
    \nabla_Y h &=& \rho_1(Y S X^TX S^T  + (Q - S) X S^T) + \rho_2(Y - X) 
    \\
                         && +  \mu Y X^T X,
  \end{array}
\]
whose Lipschitz constants are given by
\[
  \begin{array}{l}
    L_1(Y) \,=\, \ds \| \rho_1 S^T Y^T Y S \,+\, \mu Y^T Y  \,+\,  \rho_2 I\|_2  \\[0.1cm]
    L_2(X) \,=\, \ds \| \rho_1 S X^T X S^T  \,+\, \mu X^T X \,+\, \rho_2 I \|_2.
  \end{array}
\]
Therefore, Algorithm~\ref{alg.palm} applies by modifying the computation of the Lipschitz constants.
\end{remark}

\begin{remark}[Comparison with ADMM]
\tc{black}{The alternating direction method of multipliers~(ADMM) has been a very powerful tool in distributed control and optimization~\cite{boyparchu11,honluoraz16,hajchawan16,wanyinzen18}. Since ADMM is a class of proximal algorithms~\cite{parboy13}, it is closely related to PALM. It is worth mentioning that ADMM is most useful for minimizing the sum of convex functions. For certain classes of nonconvex problems, the convergence of ADMM has been established in~\cite{honluoraz16,hajchawan16,wanyinzen18}. For the cardinality~\eqref{eq.cardc} and the rank function~\eqref{eq.rankc}, ADMM may not converge for~\eqref{eq.lc}. The solution to which ADMM converges may also depend on the value of $\rho$; see~\cite{honluoraz16}. Furthermore, efficient methods for subproblems in ADMM that deal with the Lyapunov penalty are yet to be developed.}
\end{remark}

%
%\begin{remark}[Stepsize]
%Typical descent-type methods (e.g., gradient projection) require an appropriate stepsize to ensure sufficient decrease in the objective value. In contrast, PALM requires no stepsize rule because the Lipschitz condition controls the decrease of the cost function. This point is further elaborated in Section~\ref{sec.conv}. When the computation of the gradient projection is nontrivial and when stepsize search takes several projections,  PALM often outperforms descent-type algorithms. We compare PALM and gradient projection method in Section~\ref{sec.exp}.
%\end{remark}
%linfarjov13,linjovgeo13,linche16

\section{Convergence Analysis}
\label{sec.conv}

In this section, we show that Algorithm~\ref{alg.palm} globally converges to a critical point of the nonconvex, nonsmooth problem~\eqref{eq.lcstd}. Furthermore, the objective value is monotonically decreasing throughout the PALM iterations. We build upon the seminal work on the convergence of PALM for generic problems~\cite{bolsabteb14}. Our contributions are the establishments of the required Lipschitz conditions and the KL property.

We begin with a technical lemma on the Lipschitz conditions of the objective function $\omega$.
\begin{lemma} \label{lem.pro}
  The objective function $\omega$ in~\eqref{eq.lcstd} satisfies the following properties:
  \begin{enumerate}
  \item \label{pro.lb}
   $\inf_{X,Y} \omega(X,Y) > -\infty$, $\inf_X f(X) > -\infty$, and $\inf_Y g(Y) > -\infty$.
  \item \label{pro.lip}
   For a fixed $Y$, the partial gradient $\nabla_X h(X,Y)$ is globally Lipschitz; that is, there exists $L_1(Y)$ such that
$
  \| \nabla_X h (X_1,Y) - \nabla_X h(X_2,Y) \|_F \leq L_1(Y) \| X_1 -  X_2 \|_F
$
for all $X_1$ and $X_2$. Likewise, for a fixed $X$, the partial gradient $\nabla_Y h(X,Y)$ is globally Lipschitz; that is, there exists $L_2(X)$ such that
$
  \| \nabla_Y h (X,Y_1) - \nabla_Y h(X,Y_2) \|_F \leq L_2(X) \| Y_1 -  Y_2 \|_F
$
for all $Y_1$ and $Y_2$.
\item \label{pro.lipbd}
  There exist bounded constants $q_1^-$, $q_1^+$, $q_2^-$, $q_2^+ > 0$ such that
  \begin{equation}
    \label{eq.lipbd}
    \begin{array}{c}
    \inf_k \{ L_1(Y^k) \}  \, \geq \, q_1^-
     ~~\mbox{and}~~
    \inf_k \{ L_2(X^k) \}  \, \geq \, q_2^-
      \\[0.1cm]
    \sup_k \{ L_1(Y^k) \}  \, \leq \, q_1^+
     ~~\mbox{and}~~
    \sup_k \{ L_2(X^k) \}  \, \leq \, q_2^+.
    \end{array}
  \end{equation}
\item \label{pro.lipc2}
  The entire gradient $\nabla h(X,Y)$ is Lipschitz continuous on the bounded subsets of $\R^{p \times p} \times \R^{p \times p}$.
  \end{enumerate}
\end{lemma}
\begin{proof}
Property~\ref{pro.lb} is a direct consequence of the nonnegativity of $f$ in~\eqref{eq.f}, $h$ in~\eqref{eq.H}, and the indicator function $g$ in~\eqref{eq.card} and~\eqref{eq.rank}. Property~\ref{pro.lip} follows from the Lipschitz constants derived in~\eqref{eq.lipcon}. To show property~\ref{pro.lipbd}, note that $L_1(Y)$ in~\eqref{eq.lipcon} is clearly bounded below for all $Y$. In particular,
\[
  L_1^2(Y) \,=\, \rho_1^2 \| S^T Y^T Y S \|_F^2  \,+\, 2 \rho_1 \rho_2 \| Y S \|_F^2 \,+\, \rho_2^2  \,\geq\, \rho_2^2 \, > \, 0.
\]
On the other hand, since $Y^k$ is the minimizer of a feasible problem over a bounded set, it is bounded for all $k$ and hence $L_1(Y^k)$ is bounded above. Thus, the entire sequence $L_1(Y^k)$ satisfies the upper and lower bounds in~\eqref{eq.lipbd}. An analogous argument shows that the Lipschitz constant $L_2(X)$ satisfies~\eqref{eq.lipbd}. Property~\ref{pro.lipc2} is a direct consequence of the twice continuous differentiability of $h$ and the mean value theorem.
\end{proof}

A few comments are in order. Property~\ref{pro.lb} ensures that each proximal operator in PALM is well defined, as well as the minimization of $\omega$. Property~\ref{pro.lip} on the boundedness of the Lipschitz constants is critical for convergence. Note that the block-Lipschitz property in $X$ and $Y$ is weaker than standard assumptions in proximal methods that require $\omega$ to be globally Lipschitz in {\em joint\/} variables $(X,Y)$. Property~\ref{pro.lipbd} guarantees that the Lipschitz constants for the partial gradients are lower and upper bounded by finite numbers. Property~\ref{pro.lipc2} is a technical condition for controlling the distance between two consecutive steps in the sequence $(X^k,Y^k)$.

\begin{proposition}
  \label{pro.mono}
Let $Z^k \DefinedAs (X^k,Y^k)$ be a sequence generated by Algorithm~\ref{alg.palm}.  Then,
    \[
       \frac{\delta}{2} \| Z^{k+1} - Z^k \|^2_F \; < \; \omega(Z^k) \,-\, \omega(Z^{k+1}), \quad \forall k \geq 0
    \]
    where $\delta = \min \{(\gamma_1 - 1) q_1^-, (\gamma_2 -1) q_2^-\}$.
Furthermore, $\lim_{k\to \infty} \|Z^{k+1} - Z^k \|^2_F = 0$.
%\[
%\sum_{k=1}^\infty \|X^{k+1} - X^k \|^2_F 
%\,+\, 
%\| Y^{k+1} - Y^k \|^2_F  
%\,=\,
%\sum_{k=1}^\infty \| Z^{k+1} - Z^k \|^2_F  \, < \, \infty
%\]
%and thus $\lim_{k\to \infty} \|Z^{k+1} - Z^k \|^2_F = 0$.
\end{proposition}
\begin{proof}
\tc{black}{Consider the proximal operator
\[
\bfu^{k+1} \in \amin
\left\{ 
\eta(\bfu) + \frac{\tau}{2}\| \bfu - ( \bfu^{k} - \frac{1}{\tau} \nabla \bfh (\bfu^k)) \|^2 
\right\}
\]
where $\bfh$ is a continuously differentiable function with Lipschitz constant $L_\bfh$ and $\eta$ is a proper, bounded, lower semicontinuous function. Recall the sufficient decrease property of the proximal map~\cite[Lemma 3.2]{bolsabteb14} 
\begin{equation}
\label{eq.descent}
\bfh(\bfu^{k+1}) + \eta(\bfu^{k+1})
\leq \bfh(\bfu^{k}) + \eta(\bfu^{k}) - \frac{\tau - L_\bfh}{2} \| \bfu^{k+1} - \bfu^{k} \|^2.
\end{equation}
Applying~\eqref{eq.descent} to~\eqref{eq.proxf} and~\eqref{eq.proxg} yields
\begin{align*}
h(X^{k+1},Y^k) + f(X^{k+1}) 
& \leq 
h(X^{k},Y^k) + f(X^{k})  \\
& - \frac{c_k - L_1}{2} \|X^{k+1} - X^k\|_F^2 \\
h(X^{k+1},Y^{k+1}) + g(Y^{k+1}) 
& \leq 
h(X^{k+1},Y^k) + g(X^{k}) \\
& - \frac{d_k - L_2}{2} \|Y^{k+1} - Y^k\|_F^2 .
\end{align*}
Adding these two inequalities leads to 
\begin{align*}
\omega(Z^{k+1}) 
\leq 
\omega(Z^{k}) 
& - \frac{c_k - L_1}{2} \|X^{k+1} - X^k\|_F^2 \\
& - \frac{d_k - L_2}{2} \|Y^{k+1} - Y^k\|_F^2.
\end{align*}
Since $c_k = \gamma_1 L_1$ and $d_k=\gamma_2 L_2$, we obtain
\begin{align*}
\omega(Z^k) - \omega(Z^{k+1})
& \geq 
\frac{(\gamma_1 - 1) L_1}{2} \|X^{k+1} - X^k\|_F^2 \\
& + \frac{(\gamma_2 - 1) L_2}{2} \|Y^{k+1} - Y^k\|_F^2 \\
& \geq  
\frac{\delta}{2} \|Z^{k+1} - Z^k\|_F^2 
\end{align*}
where $\delta \DefinedAs \min\{(\gamma_1 - 1) q_1^-, (\gamma_2 - 1) q_2^-\} $ and $q_1^-,q_2^-$ are the lower bounds of Lipschitz constants defined in~\eqref{eq.lipbd}. Since $\omega$ is bounded below and $\delta$ is strictly positive, it follows that $\lim_{k\to \infty} \|Z^{k+1} - Z^k \|^2_F = 0$. This completes the proof.}
\end{proof}

Proposition~\ref{pro.mono} guarantees that the objective value is monotonically decreasing and the PALM algorithm is globally convergent. Note that $\delta > 0$ throughout iterations because $\gamma_1,\gamma_2 > 1$ (see Algorithm~\ref{alg.palm}) and $q_1^-, q_2^- > 0$ (see Lemma~\ref{lem.pro}). The convergence of the decision variable $Z^k$ can be measured by the convergence of the objective value. The numerical experiments in Section~\ref{sec.exp} verify this convergence behavior.

We next show that Algorithm~\ref{alg.palm} converges to a critical point of $\omega$.\footnote{For nonconvex, nonsmooth functions, the critical point is understood as the points whose Frechet subdifferential contains $0$.} The key step is to establish the KL property of $\omega$.
\begin{definition}[KL property~\cite{bolsabteb14}]
Let $\bff: \R^d \to (-\infty,+\infty]$ be proper and lower semicontinuous. The function $\bff$ is said to have the {\em Kurdyka-Lojasiewicz (KL)\/} property at $\bar{\bfu} \in \mbox{dom} \, \partial \bff \DefinedAs \{ \bfu \in \R^d : \partial \bff(\bfu) \neq \emptyset  \}$ if there exist $\eta \in (0, +\infty]$, a neighborhood ${\cal N}$ of $\bar{\bfu}$, and a scalar-valued function $\psi$ such that for all
$
  \bfu \in {\cal N} \cap \{ \bff(\bar{\bfu}) < \bff(\bfu) < \bff(\bar{\bfu}) + \eta \},
$
the following inequality holds:
$
  \psi'(\bff(\bfu) - \bff(\bar{\bfu})) \cdot \mbox{dist}(0,\partial \bff(\bfu)) \,\geq\, 1,
$
where $()'$ denotes the derivative function and $\mbox{dist}(x,s) \DefinedAs \inf \{ \|y-x\| : y \in {\bf s} \}$ denotes the distance from a point $x \in \R^d$ to a set ${\bf s} \subset \R^d$. A function $\bff$ is called a KL function if $\bff$ satisfies the KL property at each point of the domain of the gradient $\partial \bff$.
\end{definition}

While KL property is a technical condition, it is shown in~\cite{bolsabteb14} that a large class of nonsmooth problems that arise in modern applications satisfy the KL property. For the low-complexity autoregressive model~\eqref{eq.lcstd}, the concept of semi-algebraic function is instrumental in establishing the KL property.

\begin{definition}[Semi-algebraic function~\cite{bolsabteb14}]
A subset ${\cal S}$ of $\R^d$ is a real {\em semi-algebraic set\/} if there exists a finite number of real polynomial functions ${\bf g}_{ij}$ and ${\bf h}_{ij} : \R^d \to \R$ such that
$
{\cal S} \,=\, \bigcup_{j=1}^p \bigcap_{i=1}^q \{  \bfu \in \R^d: {\bf g}_{ij}(\bfu) = 0 ~\mbox{and}~ {\bf h}_{ij}(\bfu) < 0  \}.
$
A function $\bfh:\R^d \to (-\infty,+\infty]$ is called semi-algebraic function if its graph $\{(\bfu,v) \in \R^{d+1} : \bfh(\bfu) = v \}$ is a semi-algebraic subset of $\R^{d+1}$.
\end{definition}

A proper, lower semicontinuous, and semi-algebraic function satisfies the KL property; see~\cite[Theorem 5.1]{bolsabteb14}. Based on this result, we now show the KL property of $\omega$.

\begin{lemma} \label{lem.kl}
  The objective function $\omega$ in~\eqref{eq.lcstd} satisfies the KL property.
\end{lemma}
\begin{proof}
Since $\omega$ is the summation of smooth functions $f$, $h$ and the indicator function $g$ that is lower semicontinuous, it follows that $\omega$ is a proper and lower semicontinuous function. To show that it is a semi-algebraic function, we examine each term in $\omega$. Clearly, $f$ and $h$ are semi-algebraic because they are real-valued polynomials. Moreover, the indicator function of the semi-algebraic set $\{Y \,|\, \card(Y) \leq s\}$ is semi-algebraic, and the indicator function of the semi-algebraic set $\{Y \,|\, \rank(Y) \leq r\}$ is also semi-algebraic; see~\cite{bolsabteb14}. A finite sum of semi-algebraic functions is semi-algebraic. This completes the proof.
\end{proof}

\tc{black}{We conclude this section by invoking the convergence result~\cite[Theorem 3.1]{bolsabteb14} of PALM for KL functions.
\begin{proposition}
  \label{pro.finite}
  Let $Z^k = (X^k,Y^k)$ be a sequence generated by the PALM algorithm. Suppose that $\omega$ is a KL function that satisfies the properties in Lemma~\ref{lem.pro}. Then the sequence $\{Z^k\}$ converges to a critical point $Z^* \;=\; (X^*,Y^*)$ of $\omega$.
\end{proposition}}

%  \begin{enumerate}
%    \item The sequence $\{Z^k\}$ has a finite length, that is, 
%    \[
%        \sum_{k=1}^\infty \| Z^{k+1} - Z^k \|_F \; < \; \infty.
%    \]
%    \item The sequence $\{Z^k\}$ converges to a critical point $Z^* \;=\; (X^*,Y^*)$ of $\omega$.
%  \end{enumerate}
%\begin{proof}
%Apply our Lemma~\ref{lem.pro} and Lemma~\ref{lem.kl} to Theorem 3.1 of~\cite{bolsabteb14}.
%\end{proof}

\section{Numerical Experiments}
\label{sec.exp}

In this section, we evaluate the performance of Algorithm~\ref{alg.palm} on both synthetic and real-world data. We demonstrate that the solution converges to a matrix with the prescribed level of nonzero elements or matrix rank. Furthermore, the objective value (i.e., estimation error) decreases monotonically as predicted by the convergence analysis. 

We also compare the estimation errors of the autoregressive models obtained from nonconvex and convex constraints; in particular, we focus on the cardinality constraint versus the $\ell_1$ constraint. Our numerical results show that the cardinality constraint achieves a smaller error than the $\ell_1$ constraint on a variety of systems drawn from the COMPLeib library~\cite{lei04,leilip04}. Moreover, we show that PALM outperforms with gradient projection method when handling the $\ell_1$ constraint.

In our experiments, we assume that the covariance matrix of the noise $\epsilon(t)$ is $Q = \sigma^2 I$. We set $\gamma_1 = \gamma_2 = 2$ in Algorithm~\ref{alg.palm}. The hyperparameters $\rho_1$ and $\sigma$ are determined through cross validation.

\subsection{Synthetic Data}

We test the performance of the proposed method on a sparse example and a low-rank example with synthetic transition matrices of size $200 \times 200$. In both examples, we use time series of length $n=50$ for training and $m=800$ for testing. For steady-state data, we set the length $N=1600$. The performance of the identified autoregressive model is evaluated by using the normalized error and the cosine score proposed in~\cite{huasch11}
\begin{align*}
\text{Normalized error: }  \frac{1}{m-1}\sum_{t=1}^{m-1}\frac{\|\phi(t+1) - \widehat{A} \phi(t)\|}{\| \phi(t+1) -  \phi(t)\|}\\
\text{Cosine score: } 
 \frac{1}{m-1}\sum_{t=1}^{m-1}\frac{|( \phi(t+1) -  \phi(t))^T( \phi(t ) - \widehat{A} \phi(t))|}{ \| \phi(t+1) - \phi(t)\|\| \phi(t) - \widehat{A} \phi(t)\|}.
\end{align*}
 A smaller normalized error (lower bounded by 0) and a higher cosine score (upper bounded by 1) imply better performance.

\subsubsection{Sparse Example}
The sparse matrix is generated by using the rule
$
  A = (0.95 M)/\max_k (|\lambda_k(M)|),
$
where $M$ has $5000$ normally distributed nonzero elements and $\lambda_k(M)$ denotes the eigenvalues of $M$. We set $s = 5000$ in the cardinality constraint~\eqref{eq.cardc}.

Figure~\ref{fig.sparse} shows the convergence results. The objective value monotonically decreases, as Proposition~\ref{pro.mono} indicates. The errors in two consecutive steps, namely,  
$e^k_X \,=\, \| X^{k+1} - X^k \|_F$,
$e^k_Y \,=\, \| Y^{k+1} - Y^k \|_F$, 
$e^k_{XY} \,=\, \| X^{k} - Y^k \|_F$,
all decrease quickly. It takes fewer than $30$ iterations to reach $e_X,e_Y \leq 10^{-4}$ and $e_{XY} \leq 3.5 \times 10^{-4}$. Note that the solution has exactly $5000$ nonzero elements as required by the cardinality constraint. For the estimated matrix $\widehat{A}$, the normalized error is $0.2848$ and the cosine score is $0.9582$.

\begin{figure*}
  \centering
  \begin{tikzpicture}
    \begin{semilogyaxis} [width=0.35\textwidth,
      xlabel = PALM iteration index $k$,
      mark repeat={5},
      title = Objective value \mbox{$\omega(X^k,Y^k)$}
      ]
      \addplot table[x=Iter,y=obj] {cardp200_v2.txt};
    \end{semilogyaxis}
  \end{tikzpicture}
  \begin{tikzpicture}
    \begin{semilogyaxis} [width=0.35\textwidth,
      xlabel = PALM iteration index $k$,
      mark repeat={5},
      title = Errors in consecutive steps,
      ]
      \addplot table[x=Iter,y=diffX] {cardp200_v2.txt};
      \addplot table[x=Iter,y=diffY] {cardp200_v2.txt};
      \addplot table[x=Iter,y=diffXY] {cardp200_v2.txt};
      \legend{$e^k_X$,$e^k_Y$,$e^k_{XY}$}
    \end{semilogyaxis}
  \end{tikzpicture}
  \begin{tikzpicture}
    \begin{axis}[width=0.35\textwidth,
      xlabel=$5000$ nonzero entries,
      title=Sparsity pattern of $\widehat{A}$]
      \addplot graphics [
      xmin=0,xmax=200,ymin=0,ymax=200,
      includegraphics={trim=12 9 12 8,clip}] {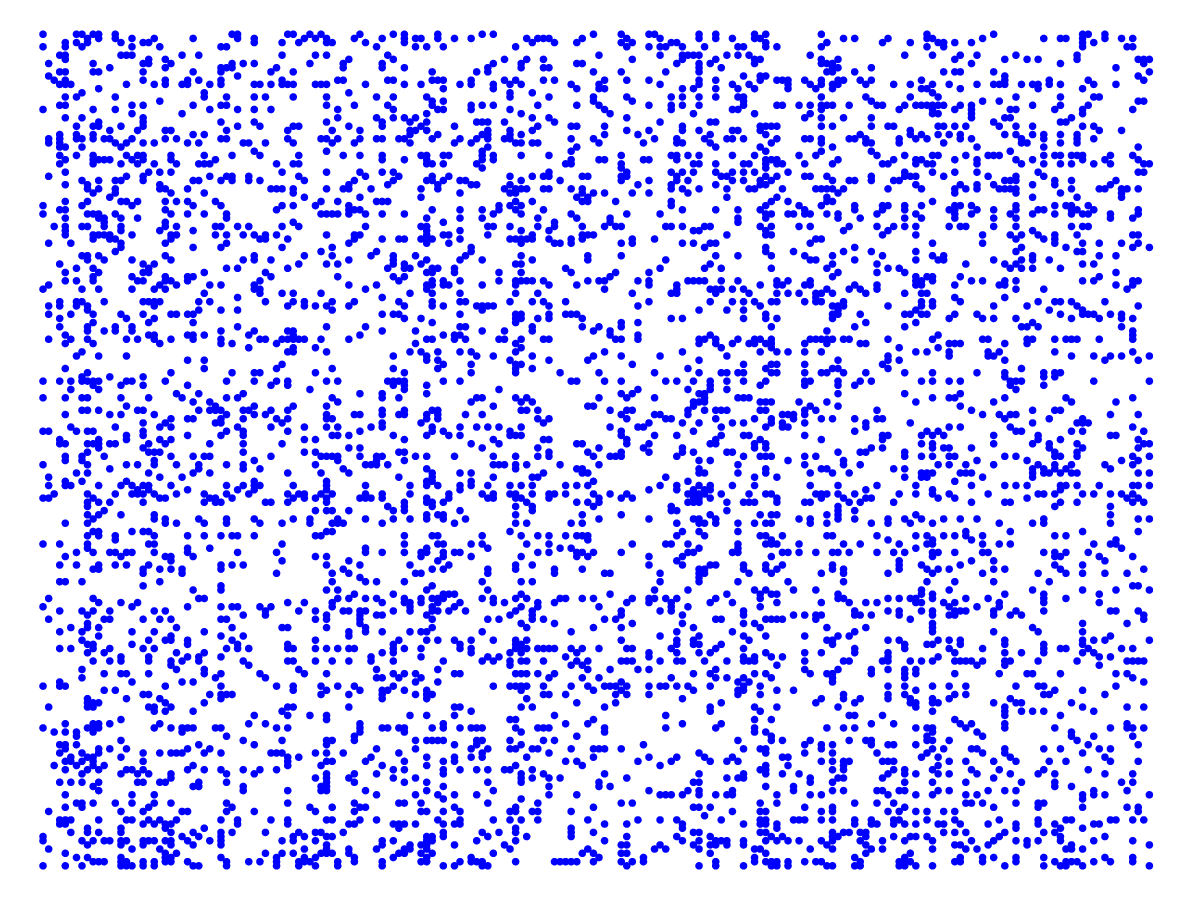};
    \end{axis}
  \end{tikzpicture}
  \caption{Convergence results of PALM for the sparse example: the objective value (left), the errors in consecutive steps (middle), and the sparse solution with $5000$ nonzero entries (right).}
  %\caption{Convergence results of PALM for the sparse example: the objective value (left) and the errors in consecutive steps (right).}
  \label{fig.sparse}
\end{figure*}
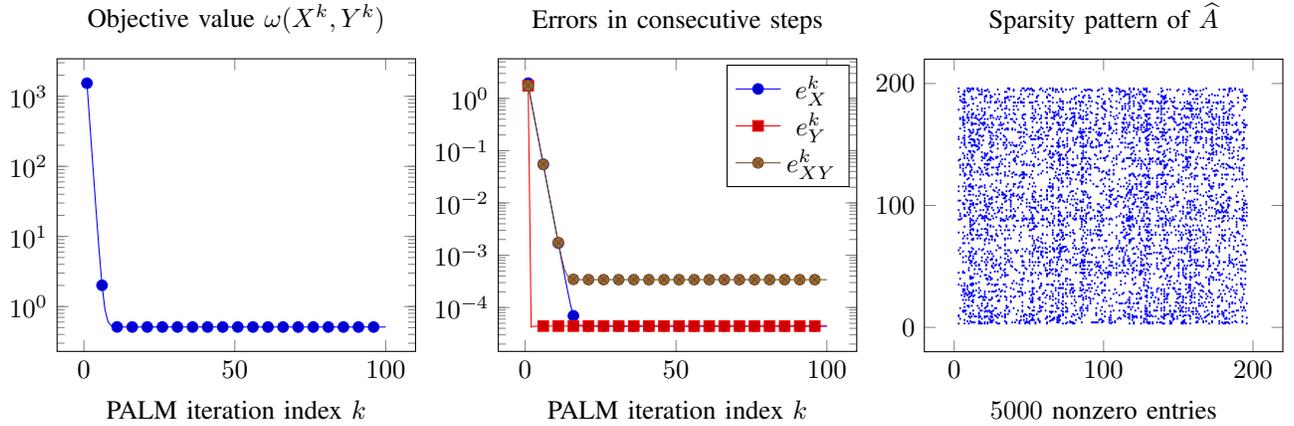

\subsubsection{Low-Rank Example}

The low-rank matrix is generated by using the rule $A = {\cal U} \Sigma {\cal V}$, where $\Sigma \in \R^{25 \times 25}$ is a diagonal matrix with random diagonal entries uniformly distributed in $[0,1)$, and ${\cal U} \in \R^{200 \times 25}$ and ${\cal V} \in \R^{25 \times 200}$ are random orthonormal matrices. By construction $A \in \R^{200 \times 200}$ is stable with $\rank(A) = 25$. We set $r = 25$ in the rank constraint~\eqref{eq.rankc}.

Figure~\ref{fig.lowrank} shows the convergence results. Similar to those for the sparse example in Figure~\ref{fig.sparse}, we observe that the objective value $\omega$ monotonically decreases and the errors in two consecutive steps decrease quickly. It takes fewer than $30$ iterations to reach $e_X,e_Y \leq 3 \times 10^{-5}$ and $e_{XY} \leq 2 \times 10^{-4}$. The solution has a numerical rank $25$, as required by the rank constraint. For the estimated matrix $\widehat{A}$, the normalized error is $0.6949$ and the cosine score is $0.7189$.

\begin{figure*}[t]
  \centering
  \begin{tikzpicture}
    \begin{semilogyaxis} [width=0.35\textwidth,
      xlabel = PALM iteration index $k$,
      mark repeat={5},
      title = Objective value \mbox{$\omega(X^k,Y^k)$}
      ]
      \addplot table[x=Iter,y=obj] {rankp200_v2.txt};
    \end{semilogyaxis}
  \end{tikzpicture}
  \begin{tikzpicture}
    \begin{semilogyaxis} [width=0.35\textwidth,
      xlabel = PALM iteration index $k$,
      mark repeat={5},
      title = Errors in consecutive steps,
      ]
      \addplot table[x=Iter,y=diffX] {rankp200_v2.txt};
      \addplot table[x=Iter,y=diffY] {rankp200_v2.txt};
      \addplot table[x=Iter,y=diffXY] {rankp200_v2.txt};
      \legend{$e^k_X$,$e^k_Y$,$e^k_{XY}$}
    \end{semilogyaxis}
  \end{tikzpicture}
  \begin{tikzpicture}
    \begin{semilogyaxis} [width=0.35\textwidth,
      xlabel = Eigenvalue index $i$,
      mark repeat={5},
      title = $|\lambda_i(X)|$
      ]
      \addplot table[x=Index,y=eigAest] {rankp200Eig_v2.txt};
    \end{semilogyaxis}
  \end{tikzpicture}
  \caption{Convergence results of PALM for the low-rank example: the objective value (left), the errors in consecutive steps (middle), and the low-rank solution with $25$ nonzero eigenvalues (right).}
  %\caption{Convergence results of PALM for the low-rank example: the objective value (left) and the errors in consecutive steps (right).}
  \label{fig.lowrank}
\end{figure*}
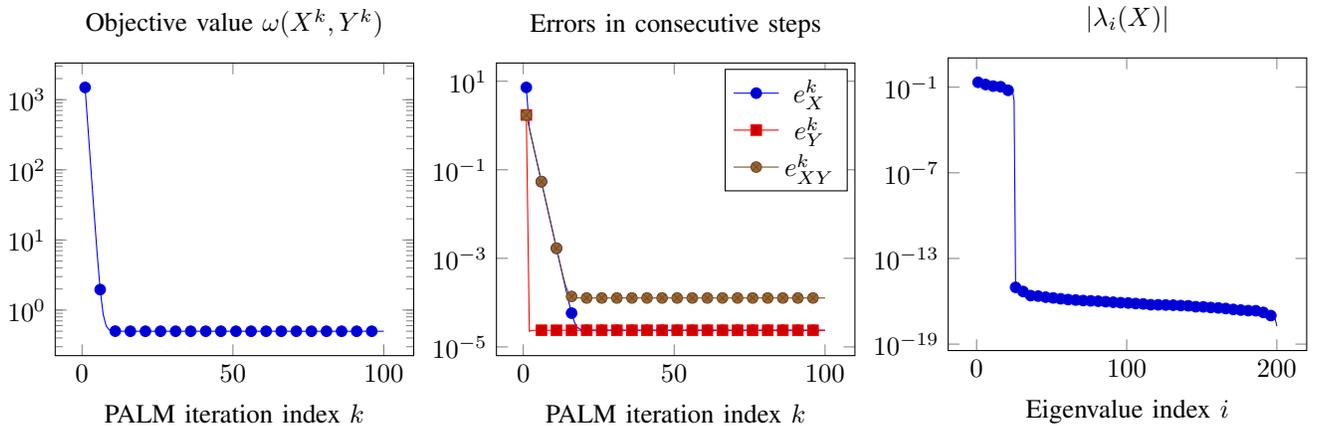

\subsection{Electricity Load Data}

We explore the utility of the low-complexity models on an electricity load data set from the UCI repository.\footnote{\url{http://archive.ics.uci.edu/ml/datasets/ElectricityLoadDiagrams20112014}} The data set consists of 15-minute interval load readings of clients over 1461 days. To investigate the daily dynamics, we aggregate the data in every 24-hour interval. Because over half of the clients are not registered in the first year, we start from the second-year data and collect clients whose time series data are uninterrupted. Interruptions may arise from late registration of clients, missing data, or a period of low electricity consumption due to inactivity. Such a preprocessing results in $272$ clients and $1095$ daily readings per client. We further subtract each time series by its seasonal mean, that is, the mean of the same day along all the years, and normalize it by the standard deviation.

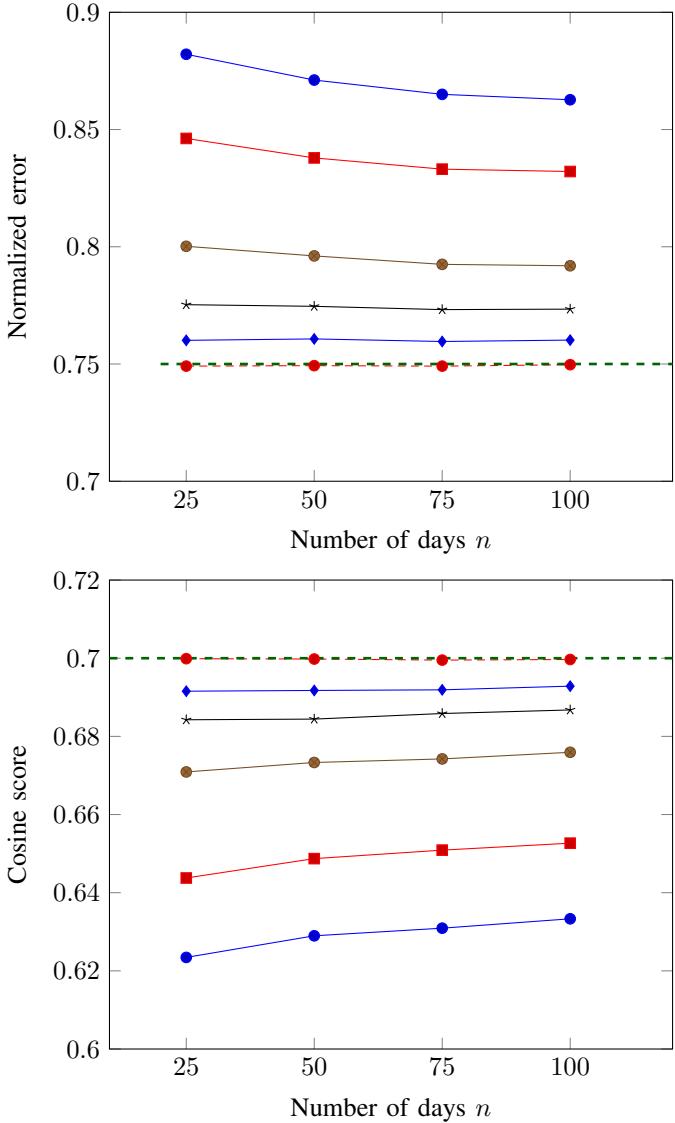
\begin{figure}
  \centering
  \begin{tikzpicture}
    \begin{axis} [width=0.5\textwidth,
      xlabel = Number of days $n$,
      ylabel = Normalized error,
      ymin=0.70,ymax=0.9,xmin=10,xmax=120,
      xtick = {25,50,75,100}
      ]
      \addplot table[x=n,y=card100] {elec_card_normalized_error.txt};
      \addplot table[x=n,y=card125] {elec_card_normalized_error.txt};
      \addplot table[x=n,y=card150] {elec_card_normalized_error.txt};
      \addplot table[x=n,y=card175] {elec_card_normalized_error.txt};
      \addplot table[x=n,y=card200] {elec_card_normalized_error.txt};
      \addplot table[x=n,y=cardnosp] {elec_card_normalized_error.txt};
      \addplot +[dashed,mark=none,darkgreen,line width = 1pt] coordinates {(20, 0.75) (120, 0.75)};
    \end{axis}
  \end{tikzpicture}
  \begin{tikzpicture}
    \begin{axis} [width=0.5\textwidth,
      xlabel = Number of days $n$,
      ylabel = Cosine score,
      ymin=0.6,ymax=0.72,xmin=10,xmax=120,
      xtick = {25,50,75,100}
      ]
      \addplot table[x=n,y=card100] {elec_card_cos_score.txt};
      \addplot table[x=n,y=card125] {elec_card_cos_score.txt};
      \addplot table[x=n,y=card150] {elec_card_cos_score.txt};
      \addplot table[x=n,y=card175] {elec_card_cos_score.txt};
      \addplot table[x=n,y=card200] {elec_card_cos_score.txt};
      \addplot table[x=n,y=cardnosp] {elec_card_cos_score.txt};
      \addplot +[dashed,mark=none,darkgreen,line width = 1pt] coordinates {(10, 0.7) (120, 0.7)};
    \end{axis}
  \end{tikzpicture}
  \caption{Electricity load data: Performance comparison between the least squares estimator~\eqref{eq.ls} when $p<n$ (\textcolor{darkgreen}{--}) and the Lyapunov-penalized model~\eqref{eq.lc} when $p>n$, with different levels of cardinality: $s=100p$ (\textcolor{blue}{$\bullet$}), $s=125p$ (\textcolor{red}{\ding{110}}), $s=150p$ (\textcolor{brown}{$\bullet$}), $s=175p$ (\textcolor{black}{$*$}),  $s=200p$ (\textcolor{blue}{$\diamond$}), $s=p^2$ (\textcolor{red}{$\bullet$}).}
  \label{fig.elec}
\end{figure}

We first use the least squares estimator~\eqref{eq.ls} to obtain a reference model. To this end, the first $995$ days are used for training and the last $100$ days are used for testing. Note that the reference model provides an upper bound on the performance, because the number $n$ of measurements is sufficiently greater than the data dimension $p$.

Next, we test the low-complexity models~\eqref{eq.lc} with different thresholds for the cardinality and the rank constraint. In particular, we set $s \in \{100p,125p,150p,175p,200p,p^2\}$ and similarly $r \in \{100,125,150,175,200,p\}$. We take the first $n$ days with $n \in \{25,50,75,100\}$ as the training data, the last $100$ days as the testing data, and 600 randomly sampled days in the remaining dataset as the steady-state data. We repeat the experiment five times for each set of $s$, $r$, and $n$.

Figure~\ref{fig.elec} shows the performance measures as the number $n$ of training data and the sparsity level $s$ vary. Three observations can be made. First, the normalized error and the cosine score are not sensitive to the length of the training data, because the performance varies slightly with $n$. This fact indicates that the Lyapunov penalty as a regularization is effective. Second, as the complexity of the transition matrix increases (e.g., a larger $s$), the performance gets closer to that of the least squares estimator. Third, in the case of no constraints (i.e., $s = p^2$), the Lyapunov-penalized VAR model~\eqref{eq.hs11} performs as well as the least squares estimator~\eqref{eq.ls}. In other words, the Lyapunov-penalized VAR model with a small number of time sequence data and a large number of nonsequence data is as competitive as the least squares estimator with a large amount of time sequence data. This result demonstrates the utility of the proposed method when time sequence data is limited.

\subsection{Comparison different penalties and different methods}

%As aforementioned, PALM can handle both nonconvex constraints (i.e.,~\eqref{eq.cardc} and \eqref{eq.rankc}) and convex constraints (i.e.,~\eqref{eq.l1} and \eqref{eq.nuc}). It is thus of interest to compare the performance of both classes of constraints. We test sparse models with cardinality constraint and $\ell_1$ constraint. When the $\ell_1$ constraint is employed, we also compare the performance of PALM and that of the gradient projection method. 

We test on a variety of dynamical systems from the {\em COMPleib\/} library~\cite{lei04,leilip04}. This set of systems is drawn from aircraft, helicopter, jet engine, reactor, decentralized interconnected systems, and wind energy systems. The set consists of $40$ continuous-time systems with the dimension of the state matrix $A_c \in \R^{p \times p}$ ranging from $p=3$ to $p=40$. For each model, we generate $n = p/2$ sequence data and $N = 5n$ nonsequence data for training, and $m = p$ sequence data for testing.

We use PALM to solve the problem with cardinality constraint~\eqref{eq.cardc} and $\ell_1$ constraint~\eqref{eq.l1}. For a fair comparison, the number of nonzero elements of the solution $\widehat{A}$ must be the same in both cases. For the cardinality constraint, we set the desired number of nonzeros to be $s = \alpha p^2$ for $\alpha \in \{1/2,1/4,1/8\}$. For the $\ell_1$ constraint $\|X\|_{\ell_1} \leq l$, the upper bound $l$ that yields the desired number of nonzero elements is unknown a priori. To find the matching $l$, we use  a   bisection method: Starting from an interval $[l_{\rm low}, l_{\rm up}]$ that contains the unknown $l$, repeatedly solve~\eqref{eq.lc} and divide the interval by half, until  $l$ for the desired number of nonzero elements is found or the interval is sufficiently small. We also use gradient projection~(GP) to solve~\eqref{eq.lc} with the $\ell_1$ constraint. 

Table~\ref{tab.cardl1} shows the performance of PALM-card, PALM-$\ell_1$, and  GP-$\ell_1$ methods. PALM-card outperforms the other two approaches in achieving a smaller normalized error and a higher cosine score. The percentage of cases where PALM-card outperforms the others increases with the level of sparsity, from $ 62.5\%$  for $\alpha = 1/4$ to $82.5\%$ for $\alpha=1/8$. Similarly PALM-card yields the highest cosine score in $60\%$ of the systems when $\alpha=1/4$ and in $77.5\%$ of the systems when $\alpha = 1/8$. When the $\ell_1$ constraint is used, PALM outperforms GP when $\alpha =1/2$ and $\alpha = 1/4$. 

Figure~\ref{fig.compleib} shows the normalized error and the cosine score for the three methods when $\alpha = 1/4$. Note that for 14 test problems, the errors resulted from GP-$\ell_1$ is at least two times (and  up to 43 times) of the errors from PALM-card and PALM-$\ell_1$. Similar observations can be made for the cosine score. For 12 test problems, PALM-card and PALM-$\ell_1$ result in cosine scores that are at least twice of those obtained from GP-$\ell_1$. These results suggest that the cardinality constraint is more effective than the $\ell_1$ constraint and that PALM outperforms GP by converging to better solutions.

As mentioned earlier, PALM requires no tuning for the stepsize in contrast to GP. This feature makes PALM computationally more efficient when the projection onto the constraint set becomes nontrivial. For the projection onto the $\ell_1$-ball, it turns out that the most time-consuming computation in GP is to compute the stepsize by using the Armijo rule along the projection-arc~\cite{ber99}. This is because GP requires a number of $\ell_1$ projections to compute the stepsize. As a consequence,  for the $\ell_1$ constraint PALM is computationally more efficient than GP.

\begin{table*}
  \centering
  \caption{Performance of PALM-card, PALM-$\ell_1$, and GP-$\ell_1$ on COMPLeib test problems. The sparsity level is indicated by $\alpha = s / p^2$. The table shows the number of test problems where each method outperforms the other two. For example, when $\alpha=1/4$, PALM-card achieves the smallest normalized error in 30 test problems and  the highest cosine score in 27 test problems.}
  \begin{tabular}{c|ccc|ccc|}
  \cline{2-7}
             & \multicolumn{3}{|c|}{Normalized error}                         & \multicolumn{3}{|c|}{Cosine score}  \\
  \hline
    \multicolumn{1}{|c|}{$\alpha$} & PALM-card & PALM-$\ell_1$ & GP-$\ell_1$ & PALM-card & PALM-$\ell_1$ & GP-$\ell_1$  \\
    \hline
    \multicolumn{1}{|c|}{$1/2$}    & 25        & 12             & 3           & 24 & 12 & 4\\
    \multicolumn{1}{|c|}{$1/4$}    & 30        & 6             & 4           & 27 & 10 & 3 \\
    \multicolumn{1}{|c|}{$1/8$}    & 33        & 2             & 5           & 31 & 3  & 6 \\
    \hline
  \end{tabular}
  \label{tab.cardl1}
\end{table*}

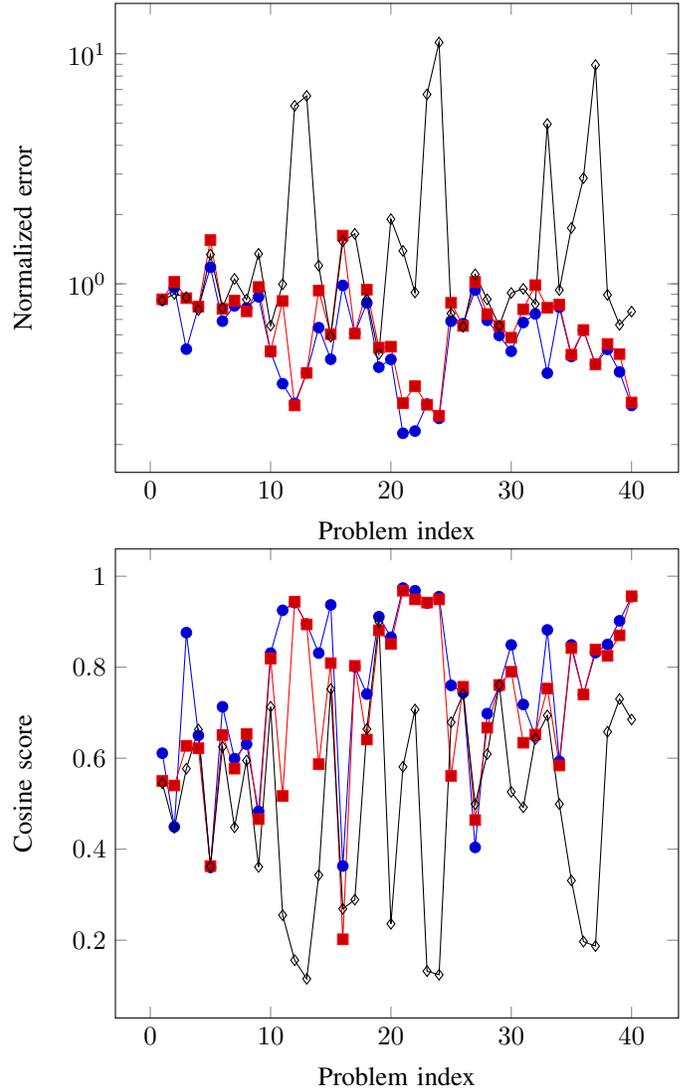
\begin{figure}
  \centering
  \begin{tikzpicture}
    \begin{semilogyaxis} [width=0.5\textwidth,
      xlabel = Problem index,
      ylabel = Normalized error
      ]
      \addplot table[x=index,y=PALMcard] {Compleib_normalized_err_4.txt};
      \addplot table[x=index,y=PALML1] {Compleib_normalized_err_4.txt};
      \addplot[mark=diamond] table[x=index,y=ProjL1] {Compleib_normalized_err_4.txt};
    \end{semilogyaxis}
  \end{tikzpicture}
  \begin{tikzpicture}
    \begin{axis} [width=0.5\textwidth,
      xlabel = Problem index,
      ylabel = Cosine score
      ]
      \addplot table[x=index,y=PALMcard] {Compleib_cos_score_4.txt};
      \addplot table[x=index,y=PALML1] {Compleib_cos_score_4.txt};
      \addplot[mark=diamond] table[x=index,y=ProjL1] {Compleib_cos_score_4.txt};
    \end{axis}
  \end{tikzpicture}
  \caption{Performance of PALM-card (\tc{blue}{$\bullet$}), PALM-$\ell_1$ (\tc{red}{\ding{110}}), and GP-$\ell_1$ ($\diamond$) on COMPLeib test problems, with sparsity level $\alpha = 1/4$.}
  \label{fig.compleib}
\end{figure}

\section{Conclusions}
\label{sec.concl}

We estimate the state transition matrix of a vector autoregressive model, with limited time sequence data but abundant nonsequence steady-state data. To reduce the complexity of the model, we propose imposing a cardinality or a rank constraint on the transition matrix. We develop the PALM algorithm to solve the resulting nonconvex, nonsmooth problem and establish its global convergence to a critical point. Numerical experiments empirically verify the convergence and demonstrate the advantage of PALM over the gradient projection method.

Several directions may be pursued following this work. First, we observe a linear convergence of the algorithm (e.g., Fig.~\ref{fig.sparse} and Fig.~\ref{fig.lowrank}). We intend to investigate the convergence rate theoretically. Second, the identified model is only one of many legitimate models that explain the given data. It is thus of interest to understand under what conditions the low-complexity model is asymptotically consistent with the ground truth, if it is sparse or low-rank in the first place. Third, while the VAR model itself has low complexity, the optimization algorithm still requires storage and computation with $p\times p$ matrices. When $p$ is too large, it would be interesting to investigate methods that reduce the cost through approximately updating the unknowns (for example, in the rank-constraint case, randomized SVD is more efficient than standard SVD).

\section*{Acknowledgments}

We thank the reviewers for constructive comments that improve this work. F. Lin is supported in part by the U.S. Department of Energy, Office of Science, Office of Advanced Scientific Computing Research, Applied Mathematics program under contract number DE-AC02-06CH11357. J. Chen is supported in part by XDATA program of the Defense Advanced Research Projects Agency (DARPA), administered through Air Force Research Laboratory contract FA8750-12-C-0323.

%\section{Projection onto $\ell_1$-ball}
%\label{sec.l1proj}
%
%Consider the problem
%\[
%\begin{array}{ll}
%\ds 
%\mini_{w \in \R^n} & 
%\ds 
%\frac{1}{2} \| w - v \|_2^2 \\
%\st         & \| w \|_{\ell_1} \, \leq \, z
%\end{array}
%\]
%where $v \in \R^n$ and $z > 0$. Algorithm~\ref{alg.l1proj} finds the optimal $w$~\cite{ducshasin08}.
%\begin{algorithm}[htb]
%   \caption{Projection onto $\ell_1$-ball~\cite{ducshasin08}.}
%   \label{alg.l1proj}
%\begin{algorithmic}
%   \STATE Let $u=|v|$, where the absolute-value sign is elementwise.
%   \STATE Sort the elements of $u$ in the descending order to get $\bar{u}$.
%   \STATE Find the index such that 
%    \[ \gamma \,=\, \max \left\{ j \,\,\Bigg\vert\,\, \bar{u}_j \,-\, \frac{1}{j} \left(\sum_{i=1}^j \bar{u}_i - z\right)  \,>\, 0 \right\}. \]
%   \STATE Define the threshold $\theta = \frac{1}{\gamma} (\sum_{i=1}^\gamma \bar{u}_i - z)$.
%   \STATE Obtain the projection $w$ with $w_i = \max \{v_i - \theta, 0\}$.
%\end{algorithmic}
%\end{algorithm}

%% `Elsevier LaTeX' style
%\section*{References}

\bibliographystyle{IEEEtran}
\bibliography{VAR}

\end{document}